\documentclass[reqno, 12pt]{amsart}
\pdfoutput=1
\makeatletter
\let\origsection=\section \def\section{\@ifstar{\origsection*}{\mysection}} 
\def\mysection{\@startsection{section}{1}\z@{.7\linespacing\@plus\linespacing}{.5\linespacing}{\normalfont\scshape\centering\S}}
\makeatother        
\usepackage{amsmath,amssymb,amsthm}    
\usepackage{mathabx}\changenotsign    
\usepackage{mathrsfs}
\usepackage{bbm, accents} 

\usepackage{xcolor}  	
\usepackage[backref]{hyperref}
\hypersetup{
	colorlinks,
    linkcolor={red!60!black},
    citecolor={green!60!black},
    urlcolor={blue!60!black},
}

\usepackage{bookmark}

\usepackage[abbrev,msc-links,backrefs]{amsrefs} 
\usepackage{doi}

\renewcommand{\PrintDOI}[1]{\doi{#1}}

\usepackage[T1]{fontenc}
\usepackage{lmodern}
\usepackage{microtype}

\usepackage[english]{babel}

\theoremstyle{plain}
\newtheorem{fact}{Fact}[section]
\newtheorem{thm}[fact]{Theorem}

\newtheorem{cor}[fact]{Corollary}

\newtheorem{lem}[fact]{Lemma}

\theoremstyle{definition}

\usepackage{geometry}
\numberwithin{equation}{section}

\let\theta=\vartheta
\let\rho=\varrho
\let\phi=\varphi

\usepackage{enumitem}

\def\alabel{\upshape({\itshape \alph*\,})}

\def\qqand{\qquad\text{and}\qquad}

\newcommand{\symb}[2]{(#1\,|\,#2)}

\begin{document}
\title[On Kemnitz' Conjecture Concerning Lattice Points in the Plane]
{On Kemnitz' Conjecture Concerning Lattice Points in the Plane}

\author[Christian Reiher]{Christian Reiher}
\address{Fachbereich Mathematik, Universit\"at Hamburg, Hamburg, Germany}
\email{Christian.Reiher@uni-hamburg.de}

\begin{abstract}
In 1961, P. Erd\H os, A.~Ginzburg, and A.~Ziv proved a remarkable theorem stating 
that each set of $2n-1$ integers contains a subset of size $n$, the sum of whose 
elements is divisible by $n.$ We will prove a similar result for pairs of integers, 
i.e., planar lattice points, usually referred to as Kemnitz' conjecture.
\end{abstract}


\maketitle

\section{Introduction}

Denoting by $f(n, k)$ the minimal number $f,$ such that any set of $f$ lattice points 
in the $k$-dimensional Euclidean space contains a subset of cardinality $n,$ the sum 
of whose elements is divisible by $n,$ it was first proved by P.~Erd\H{o}s, A.~Ginzburg, 
and A.~Ziv~\cite{EGZ}, that $f(n, 1)=2n-1$. 

The problem, however, to determine $f(n, 2)$ turned out to be unexpectedly difficult: 
A.~Kemnitz~\cite{Kem83} conjectured  it to equal $4n-3$ and knew, (1) 
that $4n-3$ is a rather straighforward lower bound\footnote[1]{In order to prove 
$f(n, 2)>4n-4$ one takes each of the four vertices of the unit square $n-1$ times.}, 
(2) that the set of all integers $n$ satisfying $f(n, 2)=4n-3$ is closed under multiplication 
and that it therefore suffices to prove this equation for prime values of $n$, and (3) 
that his assertion was correct for $n=2, 3, 5, 7$ and, consequently, also for every 
$n$ that is expressible as a product of these numbers. 

Linear upper bounds estimating $f(p, 2)$, where $p$ denotes any prime number, 
appeared for the first time in an article by N.~Alon and M.~Dubiner~\cite{AD95} 
who proved $f(p, 2)\le 6p-5$ for all $p$ and $f(p, 2)\le 5p-2$ for large $p$. 
Later this was improved to $f(p, 2) \le 4p-2$ by L.~R\'{o}nyai~\cite{Ro00}. 

In the third section of this article we prove Kemnitz' conjecture.

\section{Preliminary Results}

\subsection*{Notational conventions} 
In the sequel the letter $p$ is always assumed to designate an odd prime number
and congruence modulo $p$ is simply denoted by ``$\equiv$''. 
Roman capital letters (such as $J, X,\ldots$) will always stand for finite sets 
of lattice points in the Euclidean plane. 
The sum of the elements of such a set, taken coordinatewise, will be indicated by a 
preposed~``$\sum$''. 
Finally the symbol $\symb{n}{X}$ expresses the number of $n$-subsets of $X$, 
the sum of whose elements is divisible by $p$.

\bigskip

All propositions contained in this section are deduced without the use of 
combinatorial arguments from the following result due to Chevalley and Warning
(see e.g.,~\cite{Schm76}).

\begin{thm}
Let $P_1, P_2, \ldots, P_m\in F[x_1, \ldots,x_n]$ be some polynomials over a 
finite field~$F$ of characteristic $p$. Provided that the sum of their degrees 
is less than $n$, the number $\Omega$ of their common zeros in $F^n$ is 
divisible by $p$.
\end{thm}

\begin{proof}
It is easy to see that 
\[ 
	\Omega\equiv\sum_{y_1, \ldots, y_n\in F}\quad\prod_{\mu=1}^{m}
	\bigl(1-P_{\mu}(y_1,\,\ldots\,y_n)^{q-1}\bigr)\,,
\]
where $q=|F|$. 
Expanding the product and taking into account that
\[
	\sum_{y\in F}y^r\equiv 0 \qquad \text{ holds whenever } 1\le r\le q-2\,,
\]
we get indeed $\Omega\equiv 0$.
\end{proof}

\begin{cor}\label{cor:1}
If $|J|=3p-3$, then $1-\symb{p-1}{J}-\symb{p}{J}+\symb{2p-1}{J}+\symb{2p}{J}\equiv 0$.
\end{cor}

\begin{proof}
Let $J=\bigl\{(a_n, b_n)\,\big|\, 1\le n\le 3p-3\bigr\}$ and apply 
the above theorem to 
\[
	\sum_{n=1}^{3p-3}x_n^{\,p-1}+x_{3p-2}^{\,p-1}\,,
	\sum_{n=1}^{3p-3}a_nx_n^{\,p-1} 
	\qqand 
	\sum_{n=1}^{3p-3}b_nx_n^{\,p-1}
\]
considered as polynomials over the field containing $p$ elements. 
Their common zeros fall into two classes depending on whether $x_{3p-2}=0$ or not. 
The first class consists of 
\[
	1+(p-1)^p\symb{p}{J}+(p-1)^{2p}\symb{2p}{J}
\]
solutions, whereas the second class includes 
\[
	(p-1)^p\symb{p-1}{J}+(p-1)^{2p}\symb{2p-1}{J}
\]
solutions. 
\end{proof}

The first of the following two assertions is proved quite 
analogously and entails the second one immediately.

\begin{cor} \label{cor:2a}
If $|J|=3p-2$ or $|J|=3p-1$, then $1-\symb{p}{J}+\symb{2p}{J}\equiv 0$.
\end{cor}

\begin{cor} \label{cor:2b}
If $|J|=3p-2$ or $|J|=3p-1$, then $\symb{p}{J}=0$ implies $\symb{2p}{J}\equiv -1$.
\end{cor}

Now we come to an important statement due to N.~Alon and M.~Dubiner~\cite{AD95}.

\begin{cor} \label{cor:3}
If $J$ contains exactly $3p$ elements whose sum is $\equiv(0,0)$, then $\symb{p}{J}>0$.
\end{cor}

\begin{proof} 
Let $\mathfrak{A}\in J$ be arbitrary. Arguing indirectly we assume that $\symb{p}{J}=0$. 
This obviously implies $\symb{p}{J-\mathfrak{A}}=0$ and owing to $|J-\mathfrak{A}|=3p-1$ 
the above Corollary~\ref{cor:2b} yields ${(2p, J-\mathfrak{A})\equiv -1}$. 
So in particular we have $\symb{2p}{J-\mathfrak{A}}>0$ and the 
condition $\sum J\equiv (0,0)$ entails indeed 
$\symb{p}{J}=\symb{2p}{J}\ge \symb{2p}{J-\mathfrak{A}}>0$.
\end{proof}

The next two statements are similar to Corollary~\ref{cor:2a} and may also be 
proved in the same manner.

\begin{cor}\label{cor:4}
If $|X|=4p-3,$ then 
\begin{enumerate}[label=\alabel]
\item\label{it:4a} $-1+\symb{p}{X}-\symb{2p}{X}+\symb{3p}{X}\equiv 0$ 
\item\label{it:4b} and $\symb{p-1}{X}-\symb{2p-1}{X}+\symb{3p-1}{X}\equiv 0$.
\end{enumerate}
\end{cor}

\begin{cor}\label{cor:5}
If $|X|=4p-3$, then $3-2\symb{p-1}{X}-2\symb{p}{X}+\symb{2p-1}{X}+\symb{2p}{X}\equiv 0$.
\end{cor}

\begin{proof} 
Corollary~\ref{cor:1} implies 
\[
	\sum_I \bigl[ 1- \symb{p-1}{I}-\symb{p}{I}+\symb{2p-1}{I}+\symb{2p}{I}\bigr] \equiv 0\,,
\]
where the sum is extended over all $I\subseteq X$ of cardinality $3p-3$.
Analysing the number of times each set is counted one obtains 
\begin{align*}
	\binom{4p-3}{3p-3}&-\binom{3p-2}{2p-2}\symb{p-1}{X}-\binom{3p-3}{2p-3}\symb{p}{X} \\
	&+ \binom{2p-2}{p-2}\symb{2p-1}{X}+\binom{2p-3}{p-3}\symb{2p}{X}\equiv 0\,.
\end{align*}
The reduction of the binomial coefficients modulo $p$ leads directly to the claim. 
\end{proof}

\section{Resolution of Kemnitz' Conjecture}

\begin{lem}
If $|X|=4p-3$ and $\symb{p}{X}=0$, then $\symb{p-1}{X}\equiv \symb{3p-1}{X}$.
\end{lem}

\begin{proof} 
Let $\chi$ denote the number of partitions $X=A\cup B\cup C$ satisfying
\[
	|A|=p-1, \qquad |B|=p-2, \qquad |C|=2p\,,
\]
and moreover
\[
	\sum A\equiv (0,0), \qquad \sum B\equiv \sum X, \qquad \sum C\equiv (0,0)\,.
\]
To determine $\chi$, at least modulo $p$, we first run through all admissible $A$ 
and employing Corollary~\ref{cor:2b} we count for each of them how many possibilities
for $B$ are contained in its complement, thus getting
\[
	\chi\equiv\sum_A\symb{2p}{X-A}\equiv\sum_A -1\equiv -\symb{p-1}{X}\,.
\]
Working the other way around we infer similarly 
\[
	\chi\equiv\sum_B\symb{2p}{X-B}\equiv\sum_{X-B} -1\equiv -\symb{3p-1}{X}\,.
\]
Therefore indeed, by counting the same entities twice, $\symb{p-1}{X}\equiv \symb{3p-1}{X}$. 
\end{proof}

\begin{thm}
Any choice of $4p-3$ lattice--points in the plane contains a subset of 
cardinality $p$ whose centroid is a lattice-point as well.
\end{thm}

\begin{proof}
Adding up the congruences obtained in the Corollaries~\ref{cor:4}\ref{it:4a}, 
\ref{cor:4}\ref{it:4b},~\ref{cor:5}, and the previous lemma one deduces 
$2-\symb{p}{X}+\symb{3p}{X}\equiv 0$. 
Since $p$ is odd, this implies that $\symb{p}{X}$ and $\symb{3p}{X}$ cannot vanish 
simultaneously which in turn yields our assertion $\symb{p}{X}\ne 0$ via 
Corollary~\ref{cor:3} 
\end{proof}

As Kemnitz~\cite{Kem83} remarked, for $p=2$ the above result is an easy consequence 
of the box-principle. Since according to fact (1) mentioned in the introduction the general 
statement $f(n, 2)=4n-3$ (for every positive integer $n$) follows immediately from the 
special case where $n$ is a prime number, we have thereby proved Kemnitz' conjecture.

\end{document}